\documentclass[reqno]{amsart}
\usepackage[utf8x]{inputenc}  

\usepackage{setspace}
\usepackage[left=3.2cm, right=3.2cm, bottom=4cm]{geometry}

\usepackage{graphicx} 
\usepackage{pgf,tikz}
\usetikzlibrary{arrows}
\usepackage{amssymb}
\usepackage{pdfsync}
\usepackage{mathrsfs}
\usepackage{hyperref} 
\usepackage{verbatim} 
\usepackage{epstopdf}
\usepackage{bbm}
\usepackage[colorinlistoftodos,prependcaption,textsize=tiny]{todonotes}
\usepackage{xargs}

\def\R{\mathbb{R}}

\def\C{\mathbb{C}}

\newcommandx{\emanuel}[2][1=]{\todo[linecolor=green,backgroundcolor=green!25,bordercolor=black,#1]{#2}}

\newcommandx{\diogo}[2][1=]{\todo[linecolor=orange,backgroundcolor=orange!25,bordercolor=orange,#1]{#2}}

\newcommandx{\mateus}[2][1=]{\todo[linecolor=blue,backgroundcolor=blue!25,bordercolor=blue,#1]{#2}}

\newcommandx{\danger}[2][1=]{\todo[linecolor=red,backgroundcolor=red!25,bordercolor=blue,#1]{#2}}

\renewcommand{\d}{\text{\rm d}}

\newtheorem{theorem}{Theorem}
\newtheorem{corollary}[theorem]{Corollary}

\newtheorem{proposition}[theorem]{Proposition}
\newtheorem{lemma}[theorem]{Lemma}

\newtheorem{question}{Question}

\makeatletter
\DeclareFontFamily{U}{tipa}{}
\DeclareFontShape{U}{tipa}{m}{n}{<->tipa10}{}
\newcommand{\arc@char}{{\usefont{U}{tipa}{m}{n}\symbol{62}}}%

\newcommand{\arc}[1]{\mathpalette\arc@arc{#1}}

\newcommand{\arc@arc}[2]{%
  \sbox0{$\m@th#1#2$}%
  \vbox{
    \hbox{\resizebox{\wd0}{\height}{\arc@char}}
    \nointerlineskip
    \box0
  }%
}
\makeatother

\numberwithin{equation}{section}

\allowdisplaybreaks

\newcommand{\intav}[1]{\mathchoice {\mathop{\vrule width 6pt height 3 pt depth  -2.5pt
\kern -8pt \intop}\nolimits_{\kern -6pt#1}} {\mathop{\vrule width
5pt height 3  pt depth -2.6pt \kern -6pt \intop}\nolimits_{#1}}
{\mathop{\vrule width 5pt height 3 pt depth -2.6pt \kern -6pt
\intop}\nolimits_{#1}} {\mathop{\vrule width 5pt height 3 pt depth
-2.6pt \kern -6pt \intop}\nolimits_{#1}}}

\newcommand{\intavl}[1]{\mathchoice {\mathop{\vrule width 6pt height 3 pt depth  -2.5pt
\kern -8pt \intop}\limits_{\kern -6pt#1}} {\mathop{\vrule width 5pt
height 3  pt depth -2.6pt \kern -6pt \intop}\nolimits_{#1}}
{\mathop{\vrule width 5pt height 3 pt depth -2.6pt \kern -6pt
\intop}\nolimits_{#1}} {\mathop{\vrule width 5pt height 3 pt depth
-2.6pt \kern -6pt \intop}\nolimits_{#1}}}

\title[Sobolev regularity of polar fractional maximal functions]{Sobolev regularity of polar fractional maximal functions}

\author[Gonz\'{a}lez-Riquelme]{Cristian Gonz\'{a}lez-Riquelme}

\address{IMPA - Instituto de Matem\'{a}tica Pura e Aplicada\\
Rio de Janeiro - RJ, Brazil, 22460-320.}

\email{cristian@impa.br}
                            
\begin{document}
\begin{abstract}We study the Sobolev regularity on the sphere $\mathbb{S}^d$ of the uncentered fractional Hardy-Littlewood maximal operator $\widetilde{\mathcal{M}}_{\beta}$  at the endpoint $p=1$, when acting on polar data. We first prove that if $q=\frac{d}{d-\beta}$, $0<\beta<d$ and $f$ is a polar $W^{1,1}(\mathbb{S}^d)$ function, we have
$$\|\nabla \widetilde{\mathcal{M}}_{\beta}f\|_q\lesssim_{d,\beta}\|\nabla f\|_1.$$
We then prove that the map $$f\mapsto \big | \nabla \widetilde{\mathcal{M}}_{\beta}f \big |$$
is continuous from $W^{1,1}(\mathbb{S}^d)$ to $L^q(\mathbb{S}^d)$ when restricted to polar data. Our methods allow us to give a new proof of the continuity of the map $f\mapsto |\nabla \widetilde{M}_{\beta}f|$ from $W^{1,1}_{\text{rad}}(\mathbb{R}^d)$ to $L^q(\mathbb{R}^d)$. Moreover, we prove that a conjectural local boundedness for the centered fractional Hardy-Littlewood maximal operator $M_{\beta}$ implies the continuity of the map $f\mapsto |\nabla M_{\beta}f|$ from $W^{1,1}$ to $L^q$, in the context of polar functions on $\mathbb{S}^d$ and radial functions on $\mathbb{R}^d$.
    
\end{abstract}
\maketitle
 
\section{Introduction}
\subsection{A brief historical perspective }
The study of maximal operators is a central theme in harmonic analysis. The most classical example is the centered Hardy-Littlewood maximal operator $M$ defined for every $f\in {L}^{1}_{loc}(\mathbb{R}^d)$ as
\begin{align*}
    Mf(x):=\sup_{r>0}\ \intav{B(x,r)}|f|,
\end{align*}
where $\intav{B}g:=\frac{\int_{B}g}{m(B)}$ and $m$ is the Lebesgue measure in $\mathbb{R}^d$. The uncentered maximal operator $\widetilde{M}$ is defined analogously, but taking the supremum over open balls containing the point $x$ instead of just the ones centered at $x$. The applications of such operators are ubiquitous in analysis. For instance, one of the cornerstones of harmonic analysis is the celebrated theorem of Hardy, Littlewood and Wiener that states that $M:L^p(\mathbb{R}^d)\to L^p(\mathbb{R}^d)$ for $p>1$  and $M:L^1(\mathbb{R}^d)\rightarrow L^{1,\infty}(\mathbb{R}^d)$ are bounded. The same holds for $\widetilde{M}$. 

In recent years, motivated by applications to potential theory, there have been considerable efforts to understand what kind of properties these maximal functions have, given some a priori conditions on the initial data. We call this topic regularity theory of maximal operators. The first result in this direction is due to Kinnunen who, in his  seminal paper \cite{kinnunenseminal}, studied the action of the Hardy-Littlewood maximal operator in $W^{1,p}(\mathbb{R}^d)$ for $p>1$. He concluded that $M:W^{1,p}(\mathbb{R}^d)\rightarrow W^{1,p}(\mathbb{R}^d)$ is bounded for $p>1$, using as a main tool the inequality (that holds pointwise a.e.)
\begin{align}\label{kinnunen}
  \big |\nabla Mf\big |\le M(|\nabla f|).  
\end{align}
This work paved the way to several contributions of many researchers in this topic and its relations to other areas, see for instance \cite{aldazcolzanilazaro,joaoollidavid,carneirorenanmateus,carneirohughes,carneirosvaiter,MR2550181,localfractionalarkive,kinnunenlindqvistcreulle,hardyspaces,nontangencialoperator,ollipoincare}

The most important open problem in this field is the following $W^{1,1}$-problem. 

\begin{question}\label{w11}
Let $f\in W^{1,1}(\mathbb{R}^d)$. Does it hold that $Mf$ is weakly differentiable and 
$\|\nabla Mf\|_1\lesssim_{d}\|\nabla f\|_1$ ?
\end{question}
The most significant difficulty here is that the Hardy-Littlewood maximal operators are not bounded in $L^1(\mathbb{R}^d)$, so inequality \eqref{kinnunen} is not enough to conclude.
This problem has been settled affirmatively only in dimension $d=1$. In the uncentered case by Tanaka \cite{tanaka} and with sharp constant $C_1=1$ by Aldaz and P\'{e}rez-L\'{a}zaro \cite{sharpconstantuncentered}. In the centered case it was proved by Kurka \cite{kurkamasterpiece}. The sharp constant in the centered case is an open problem. In dimension $d>1$, Question \ref{w11} is generally open, having been settled affirmatively only for the uncentered
operator in the radial case by Luiro \cite{Luiroradial}. This recent result inspired the study of the regularity theory of maximal functions in higher dimensions when restricted to radial data. For instance, in \cite{g-remanuel} the analogue of this result for some centered kernels is proved. 
\subsection{The Hardy-Littlewood maximal operator on $\mathbb{S}^{d}$} We now move our discussion to consider maximal operators acting on functions defined on the sphere $\mathbb{S}^{d} \subset \mathbb {R}^{d+1}$. Let us establish the basic notation to be used in this context. We let $d(\zeta,\eta)$ denote the geodesic distance between two points $\zeta,\eta \in \mathbb{S}^{d}$. Let $\mathcal{B}(\zeta,r) \subset \mathbb{S}^{d}$ be the open geodesic ball of center $\zeta \in \mathbb{S}^{d}$ and radius $\pi \ge r >0$, that is
$$\mathcal{B}(\zeta,r) = \{ \eta \in \mathbb{S}^{d} \ : \ d(\zeta,\eta) < r\},$$
and let $\overline{\mathcal{B}(\zeta,r)}$ be the corresponding closed ball. Let $\widetilde{\mathcal{M}}$ denote the uncentered Hardy-Littlewood maximal operator on the sphere $\mathbb{S}^{d}$, that is, for 
$f \in L^1(\mathbb{S}^{d})$,
\begin{align*}
\widetilde{\mathcal{M}}f(\xi) = \sup_{\{\overline{\mathcal{B}(\zeta,r)} \ : \ \xi \in \overline{\mathcal{B}(\zeta,r)}\}} \frac{1}{\sigma(\mathcal{B}(\zeta,r))}\int_{\mathcal{B}(\zeta,r)} |f(\eta)|\,\d \sigma(\eta) 
= \sup_{\{\overline{\mathcal{B}(\zeta,r)} \ : \ \xi \in \overline{\mathcal{B}(\zeta,r)}\}}\  \intav{\mathcal{B}(\zeta,r)} |f(\eta)|\,\d \sigma(\eta),
\end{align*}
where $\sigma = \sigma_d$ denotes the usual surface measure on the sphere $\mathbb{S}^{d}$. The centered version ${\mathcal{M}}$ would be defined with centered geodesic balls. Fix ${\bf e} = (1, 0,0,\ldots,0) \in \mathbb{R}^{d+1}$ to be our north pole. We say that a function $f: \mathbb{S}^{d} \to \C$ is {\it polar} if for every $\xi, \eta \in \mathbb{S}^{d}$ with $ \xi \cdot {\bf e}  =  \eta \cdot {\bf e} $ we have $f(\xi) = f(\eta)$. This is the analogue, in the spherical setting, of a radial function in the euclidean setting.  We call the subset of $W^{1,1}(\mathbb{S}^d)$ of polar functions as $W^{1,1}_{\text{pol}}(\mathbb{S}^d)$. The radial functions of $W^{1,1}(\mathbb{R}^d)$ are labeled as $W^{1,1}_{\text{rad}}(\mathbb{R}^d)$. For the basic notions about Sobolev spaces in $\mathbb{S}^d$ we refer the reader to \cite{librospheres}. 
In this context, the natural analogue of the $W^{1,1}$-problem is the following.

\begin{question}\label{w11sphere}
Let $f\in W^{1,1}(\mathbb{S}^d)$. Does it hold that $\mathcal{M}f$ is weakly differentiable and
$\|\nabla \mathcal{M}f\|_1\lesssim_{d}\|\nabla f\|_1$ ?
\end{question}
Given the geometric differences between  $\mathbb{S}^d$ and $\mathbb{R}^d$ (the lack of dilations, for instance), several techniques that are available in the euclidean case are not available in the spherical case. This fact implies that Question \ref{w11sphere} presents additional difficulties when comparing to the classical $W^{1,1}$ problem.  

The progress in this problem is restricted to the uncentered version. When working on the circle $\mathbb{S}^1$, an adaptation of the proof of Aldaz and P\'{e}rez L\'{a}zaro \cite{sharpconstantuncentered} yields ${\text {Var}} (\widetilde{\mathcal{M}}f) \leq {\text {Var}} (f)$, where ${\text{Var}} (f)$ denotes the total variation of the function $f$. Recently, Carneiro and the author \cite{g-remanuel} settled affirmatively this question in the case of polar functions. The general case remains open. 
In this work we want to continue the development of regularity theory of maximal operators in the sphere setup.

\subsection{Fractional Hardy-Littlewood maximal operator:}
For $0\le \beta<d$ we define the centered fractional Hardy-Littlewood maximal operator in $L^1_{loc}(\mathbb{R}^d)$ as 
\begin{align*}
    M_{\beta}f(x)=\sup_{r>0}\ r^{\beta}\intav{B(x,r)}|f|.
\end{align*}
By taking $\beta=0$ we plainly recover the classical one. We write $\widetilde{M}_{\beta}$ as the uncentered fractional Hardy-Littlewood maximal operator. Such fractional maximal operators have applications in potential
theory and partial differential equations. 

It can be proved that $M_\beta:L^p(\mathbb{R}^d)\rightarrow L^{\frac{dp}{d-\beta p}}(\mathbb{R}^d)$  is bounded for $p>1$, but at the endpoint $p=1$ it is unbounded. Considering this, the Sobolev regularity of such operators was first considered in \cite{Kinnunensaksmanfrac}, where is proved the boundedness from $W^{1,p}(\mathbb{R}^d)$ to $W^{1, \frac{dp}{d-\beta p}}(\mathbb{R}^d)$ for $p>1$. We define $q=\frac{d}{d-\beta}$. Then, a natural question posed by Carneiro and Madrid in \cite{Carneiromadridfrac} is the following.

\begin{question}\label{fracmax}Let $f\in W^{1,1}(\mathbb{R}^d)$, $0<\beta<d$. Does it hold that $M_{\beta}f$ is weakly differentiable and
$\|\nabla M_{\beta}f\|_q\lesssim_{d,\beta}\|\nabla f\|_1$ ?
\end{question}

The progress in this problem (for general $\beta$) is restricted to the uncentered case. In that case it has been settled affirmatively only in dimension $d=1$, according to the already mentioned work  \cite{Carneiromadridfrac}, and in the radial case \cite{boundednessradialfrac}. For $\beta\ge 1$, Question \ref{fracmax} has been settled affirmatively in every dimension due to a smoothing property in both centered and uncentered cases (see \cite{Carneiromadridfrac}). The other cases remain open problems. 

Moving our discussion to $\mathbb{S}^d$ we define the uncentered fractional Hardy-Littlewood operator for $f\in L^{1}(\mathbb{S}^d)$:
\begin{align*}
\widetilde{\mathcal{M}}_{\beta}f(\xi) &= \sup_{\{\overline{\mathcal{B}(\zeta,r)} \ : \ \xi \in \overline{\mathcal{B}(\zeta,r)},r\le \pi\}} \frac{r^{\beta}}{\sigma(\mathcal{B}(\zeta,r))}\int_{\mathcal{B}(\zeta,r)} |f(\eta)|\,\d \sigma(\eta)\\
&= \sup_{\{\overline{\mathcal{B}(\zeta,r)} \ : \ \xi \in \overline{\mathcal{B}(\zeta,r)},r\le \pi\}}\  r^{\beta}\intav{\mathcal{B}(\zeta,r)} |f(\eta)|\,\d \sigma(\eta).
\end{align*}
We propose here the analogue of the previous question in this setting.
\begin{question}\label{boundednessfrac} Let $f\in W^{1,1}(\mathbb{S}^d)$, $0<\beta<d$. Does it hold that $\widetilde{\mathcal{M}}_{\beta}f$ is weakly differentiable and $\|\nabla \widetilde{\mathcal{M}}_{\beta}f\|_q\lesssim_{d,\beta}\|\nabla f\|_1$ ?
\end{question}
As far as we are concerned there is no previous result in the direction of this problem. 
Let us notice that, for the case $\beta\ge 1$ of this question, it is not enough the argument in \cite{Carneiromadridfrac}, in fact, by imitating their arguments we get, for all nonnegative $f\in W^{1,1}(\mathbb{S}^d)$ and almost every $\xi \in \mathbb{S}^d$, the inequality  
\begin{align}\label{beta>1}
|\nabla \widetilde{\mathcal{M}}_{\beta}f|(\xi)\lesssim_{d,\beta} \widetilde{\mathcal{M}}_{\beta-1}f(\xi).    
\end{align}
Therefore, by the Sobolev embedding we get
\begin{align*}
    \|\nabla \widetilde{\mathcal{M}}_{\beta}f\|_q\lesssim_{d,\beta}\|\widetilde{\mathcal{M}}_{\beta-1}f\|_q\lesssim_{d,\beta}\|f\|_{d/(d-1)}\lesssim_{d,\beta} \|f\|_{W^{1,1}(\mathbb{S}^d)}.
\end{align*}
But since, differing from the euclidean case, we cannot avoid $\|f\|_1$ in this last expression (consider, for instance, $f$ being a positive constant), Question \ref{boundednessfrac} cannot be answered directly in this case, and remains an open problem. 

Concerning to the polar case, the difficulties that Carneiro and the author faced in \cite{g-remanuel} also appear (in different ways) when dealing with this question.
Our first theorem is to get the analogue of the main result of \cite{boundednessradialfrac} in this context. We go further in the methods already developed in \cite{g-remanuel} in order to adapt the proof of \cite{boundednessradialfrac}. We get the following.
\begin{theorem}\label{boundedness} Let $f\in W^{1,1}_{\text{pol}}(\mathbb{S}^d)$, $0<\beta<d$, and $q=\frac{d}{d-\beta}$. We have
$$\|\nabla \widetilde{\mathcal{M}}_{\beta}f\|_q\lesssim_{d,\beta} \|\nabla f\|_1.$$
\end{theorem}
Notice that case $\beta=0$ corresponds to \cite[Theorem 2]{g-remanuel}.
\subsection{Continuity of maximal operators in Sobolev spaces}\label{continuityintro}

 Observation \eqref{kinnunen} was better understood after the work of Luiro, who in his work \cite{Continuityluiro} reproved the pointwise inequality \eqref{kinnunen} by getting the identity (that holds for almost every $x\in \mathbb{R}^d$)
$$\nabla Mf(x)=\intav{B_x}\nabla |f|,$$
where $Mf(x)=\intav{B_x}|f|.$ This observation also holds for the uncentered Hardy-Littlewood maximal function and it was used in \cite{Continuityluiro} in order to prove  that the map $$f\mapsto Mf$$ is continuous from $W^{1,p}(\mathbb{R}^d)$ to $W^{1,p}(\mathbb{R}^d)$ for $p>1$. His method was later extended to a more general context in \cite{Continuityluiro2}. The continuity at the endpoint case $p=1$ remains an open problem, it has been settled affirmatively for the uncentered Hardy-Littlewood maximal operator only for $d=1$ \cite{continuityuncentereddim1}. In the centered case, the problem is currently open even for $d=1$.
Recently, there has been some progress in the continuity problem when dealing with the fractional Hardy-Littlewood maximal operator. First, it was proved by Madrid \cite{madridcontinuityfrac} that the map $$f\mapsto |\nabla \widetilde{M}_{\beta}f|$$ is continuous from $W^{1,1}(\mathbb{R})$ to $L^q(\mathbb{R})$. 
After that, Beltr\'{a}n and Madrid in \cite{continuityfractionalradial} proved that for  $\beta \ge 1$, we have that the operator $f\mapsto |\nabla M_{\beta}f|$ is continuous from $W^{1,1}(\mathbb{R}^d)$ to $L^{q}(\mathbb{R}^d)$. They also proved the analogous result for $\widetilde{M}_{\beta}$. Assuming, for $\beta>0$ the boundedness $\|\nabla M_{\beta}f\|_{q}\lesssim_{\beta}\|\nabla f\|_1$ for every $ f\in W^{1,1}(\mathbb{R})$, the centered case has been settled affirmatively for $d=1$  in \cite{continuityfractionalradial}. In higher dimensions they also proved the continuity of the map $$f\mapsto |\nabla \widetilde{M}_{\beta}f|$$ from $W^{1,1}_{\text{rad}}(\mathbb{R}^d)$ to $L^q(\mathbb{R}^d)$. 
Moving our discussion to $\mathbb{S}^d$, we notice that for $\beta \ge 1$, using inequality \eqref{beta>1} the proof of \cite{continuityfractionalradial} can be adapted to this case. 

Concerning the polar case, our second result is the following.

\begin{theorem}\label{continuity} Let $0<\beta<d$ and $q=\frac{d}{d-\beta}$. The operator $f\mapsto \big |\nabla \widetilde{\mathcal{M}}_{\beta}f\big |$ maps continuously $W_{\text {pol}}^{1,1}(\mathbb{S}^{d})$ into $L^q(\mathbb{S}^{d})$.

\end{theorem}
The proof of this fact uses different ideas than the ones contained in \cite{continuityfractionalradial}.  
Our approach seems to be quite general, and we discuss some other applications of it. It is possible, in fact, to prove \cite[Theorem 1.2]{continuityfractionalradial} by using our methods (see Section 4). Moreover, we prove that under the assumption of a local boundedness conjecture we can conclude the centered version of Theorem \ref{continuity} for both the $\mathbb{S}^d$ and the $\mathbb{R}^d$ contexts (see Subsection 4.2 and Theorem \ref{conjecturalresult}). Applications for some of the maximal functions discussed in \cite{joaoollidavid} are expected (for instance, for maximal operators associated to smooth kernels of compact support), but not proved here.    
\subsection{A word on notation} In what follows we write $A \lesssim_{d} B$ if $A \leq C B$ for a certain constant $C>0 $ that may depend on the dimension $d$. We say that $A  \simeq_{d} B$ if $A \lesssim_{d} B$ and $B \lesssim_{d} A$. If there are other parameters of dependence, they will also be indicated. The characteristic function of a generic set $H$ is denoted by $\chi_H$.
\section{Proof of Theorem \ref{boundedness}}

Recall that $\sigma$ denotes the usual surface measure on the sphere $\mathbb{S}^{d}$. We denote by $\kappa_d  = \sigma(\mathbb{S}^{d}) = 2 \pi^{(d+1)/2}/\Gamma((d+1)/2)$ the total surface area of $\mathbb{S}^{d}$. With a slight abuse of notation, we shall also write
\begin{equation}\label{Def_sigma_r}
\sigma(r) := \sigma\big(\mathcal{B}(\zeta,r)\big) = \kappa_{d-1} \int_0^r (\sin t)^{d-1}\,\d t.
\end{equation}
Throughout the following we assume, without loss of generality, that $f$ takes values in $\mathbb{R}_{\ge 0}\cup \{\infty\}$.

\subsection{Preliminaries} \label{Prelim_11:49} Recalling that ${\bf e} = (1,0,0,\ldots, 0) \in \R^{d+1}$, for $\xi \in \mathbb{S}^d$ we write
$$\cos  \theta = \xi \cdot {\bf e}$$
with $\theta \in [0,\pi]$. Note that $\theta = \theta(\xi) = d({\bf e}, \xi)$ is the polar angle. 
For the sake of simplicity we henceforth assume that $f $ belongs to the set of interest for our main theorems, that is $f\in W^{1,1}_{\text{pol}}(\mathbb{S}^d)$. We define the one dimensional version of $f$ (that we also call $f$), for $r\in [0,\pi]$, as $f(r)=f(\xi)$, where $\theta(\xi)=r.$ By \cite[Lemma 13]{g-remanuel} we know that after modifying $f$ in a set of measure zero we can assume $f$ (the one dimensional version) absolutely continuous in compacts not containing $0$ or $\pi.$ In the following we continue with this assumption.

For $f \in W^{1,1}_{\text{pol}}(\mathbb{S}^{d})$ and $\xi \in \mathbb{S}^{d}$ let us define the set ${\bf B}_\xi^{\beta}$ as the set of closed balls that realize the supremum in the definition of the maximal function (since we assume $f$ continuous outside ${\bf e}$ and $-{\bf e}$ these balls have positive radius outside these points), that is
$${\bf B}_\xi^{\beta} = \left\{\overline{\mathcal{B}(\zeta,r)}; \ \zeta \in  \mathbb{S}^{d}\, ,\,  \pi \geq r \geq 0\,,\,  \xi \in \overline{\mathcal{B}(\zeta,r)} \ : \  \widetilde{\mathcal{M}_{\beta}}f(\xi) =  r^{\beta}\intav{\mathcal{B}(\zeta,r)} f(\eta)\,\d \sigma(\eta)\right\}.$$
Observe that ${\bf B}_\xi^{\beta}$ is non-empty for $\xi\notin \{{\bf e},-{\bf e}\}$.

We are mostly interested in the case where $|\nabla \widetilde{\mathcal{M}}_{\beta}f(\xi)|\neq 0$ and that can only happen in the case where $\xi \in \partial \mathcal{B}(\zeta,r)$ for every $\mathcal{B}(\zeta,r)\in \bf{B}_{\xi}^{\beta}$  (otherwise we would have that $\xi$ is a local minimum of $\widetilde{\mathcal{M}_{\beta}}f$). Moreover, since $f$ is polar, we can conclude that $\xi$, $\zeta$ and $\bf{e}$ belong to the same great circle of $\mathbb{S}^d$, and that $\bf{e}$ is not between $\xi$ and $\zeta$. Otherwise we may rotate the ball $\mathcal{B}(\zeta,r)$ with respect to the north pole $\bf{e}$ in order to get $\bf{e}$, the new center and $\xi$ in the same great circle. The crucial observation is that in this context we would have $\xi \in \text{int}(\mathcal{B}(\zeta,r))$, reaching a contradiction.
We first state an adaptation to the sphere setup of \cite[Lemma 2.1]{continuityfractionalradial}. The proof is a straightforward adaptation, we omit it.
\begin{lemma}\label{radiusconvergence}
Let $f\in W^{1,1}_{\text{pol}}(\mathbb{S}^d)$ and $\{f_j\}_{j\in \mathbb{N}}\subset W^{1,1}_{\text{pol}}(\mathbb{S}^d)$ such that $\|f-f_j\|_{W^{1,1}(\mathbb{S}^d)}\rightarrow 0$ as $j\rightarrow \infty$. For every $\xi \in \mathbb{S}^d$, choose  $\mathcal{B}(\zeta_j,r_j)\in {\bf {B}}_{\xi,j}^{\beta}$ (where ${\bf {B}}_{\xi,j}^{\beta}$ is defined analogously to ${\bf {B}}_{\xi}^{\beta}$, for each $j\in \mathbb{N}$). Then, for a.e. $\xi$, if $(\zeta,r)$ is an accumulation point of $\{(\zeta_j,r_j)\}_{j\in \mathbb{N}}$, we have $\mathcal{B}(\zeta,r)\in \bf{B}_{\xi}^{\beta}$.
\end{lemma}
Here we state the fractional version of \cite[Lemma 5]{g-remanuel}, the proof is similar, we omit it.
\begin{lemma}\label{Lem4 - inner integral sphere}
Let $f \in W^{1,1}_{\text{pol}}(\mathbb{S}^{d})$ be a nonnegative function. Assume that $\widetilde{\mathcal{M}}_{\beta}f$ is differentiable at $\xi$ and that $\mathcal{B}(\zeta,r)=\mathcal{B} \in {\bf B}_\xi^{\beta}$. Then
$$\nabla \widetilde{\mathcal{M}}_{\beta}f(\xi)v = r^{\beta}\, \intav{\mathcal{B}} \nabla f(\eta)\big(\!-(\eta\cdot v)\xi + (\eta\cdot \xi)v\big) \,\d \sigma(\eta)$$
for every $v \in \R^{d+1}$ with $v \perp \xi$. In particular,
$$\big|\nabla \widetilde{\mathcal{M}}_{\beta}f(\xi)\big| \leq  r^{\beta}\intav{\mathcal{B}} |\nabla f(\eta)| \,\d \sigma(\eta).$$
\end{lemma}
Since $f$ is polar, we can prove that $\widetilde{\mathcal{M}}_{\beta}f$ is polar and locally Lipschitz outside the poles, so Lemma \ref{Lem4 - inner integral sphere} holds almost everywhere. The proof of this fact relies on the continuity of $f$ outside the poles, that implies that near every point the radius of the maximal function is bounded by below. We explain this further in Section 3 (Proposition \ref{radiusneedstobebig}). 

Now a comment about the weak differentiability. In \cite[Lemma 13]{g-remanuel}, Carneiro and the author stated the equivalence between $g$ polar being weakly differentiable in $\mathbb{S}^d\setminus {\{{\bf e},-{\bf e}\}}$ and $g$ being weakly differentiable in $(0,\pi)$. Moreover, we stated that if that is the case and $g$ and $\nabla g$ are locally integrable in the poles then $g$ is weakly differentiable in the sphere. This result, Sobolev embedding and the previous remark joint with Theorem \ref{boundedness}(once proved) will imply that $\widetilde{\mathcal{M}}_{\beta}f$ is weakly differentiable in $\mathbb{S}^d$ when $f\in W^{1,1}_{\text{pol}}(\mathbb{S}^d)$.

\subsection{Lipschitz case}
 We assume now that our $f \in W^{1,1}_{\text{pol}}(\mathbb{S}^d)$ is a Lipschitz function. We then have
$$|\nabla f(\xi)| =|f'(\theta)|$$
for a.e. $\xi \in \mathbb{S}^d \setminus \{{\bf e}, {\bf -e}\}$, and
$$\|\nabla f\|_{L^1(\mathbb{S}^d)} = \kappa_{d-1} \int_0^\pi |f'(\theta)| \,(\sin \theta)^{d-1}\,\d \theta.$$

\subsubsection{Estimates for large radii - preliminary lemmas}
\begin{lemma}\label{Lem_10_prep}
Let $\xi \in \mathbb{S}^d \setminus \{{\bf e}, {\bf -e}\}$ and let $\mathcal{B}(\zeta,r)  \in {\bf B}_{\xi}^{\beta}$, with $\zeta$ in the half great circle determined by ${\bf e}$, $\xi$ and $-{\bf e}$. Assume that $0 \leq \theta(\zeta) < \theta(\xi)$, that $\xi \in \partial \mathcal{B}(\zeta,r)$ and that $\widetilde{\mathcal{M}}_{\beta}f$ is differentiable at $\xi$. Then
\begin{align*}
\nabla \widetilde{\mathcal{M}}_{\beta}f(\xi)(v(\xi, {\bf e})) = r^{\beta}\frac{\sigma'(r)}{\sigma(r)} \intav{\mathcal{B}(\zeta,r)} \nabla f (\eta) (v(\eta, \zeta) )\,\frac{\sigma(d(\zeta,\eta))}{\sigma'(d(\zeta,\eta))}  \,\d\sigma(\eta)-\beta r^{\beta-1}\intav{\mathcal{B}(\zeta,r)}f,
\end{align*}
where
\begin{equation*}
v(\eta, \zeta) = \frac{\zeta - (\eta \cdot \zeta)\eta}{ |\zeta- (\eta \cdot \zeta)\eta|}
\end{equation*}
is the unit vector, tangent to $\eta$, in the direction of the geodesic that goes from $\eta$ to $\zeta$. 
In particular, since $\nabla \widetilde{\mathcal{M}}_{\beta}f(\xi)(v(\xi, {\bf e}))\ge 0$, we have 
\begin{align*}
    \big |  \nabla \widetilde{\mathcal{M}}_{\beta}f(\xi) \big |\le r^{\beta}\frac{\sigma'(r)}{\sigma(r)} \intav{\mathcal{B}(\zeta,r)} \nabla f (\eta) (v(\eta, \zeta) )\,\frac{\sigma(d(\zeta,\eta))}{\sigma'(d(\zeta,\eta))}  \,\d\sigma(\eta)=r^{\beta}\frac{\sigma'(r)}{\sigma(r)}\left(\intav{\mathcal{B}(\zeta,r)}f-\intav{\partial \mathcal{B}(\zeta,r)}f\right).
\end{align*}

\end{lemma}
\begin{proof}
Let $S$ be the great circle determined by ${\bf e}$ and $\xi$. For small $h \in \R$ we consider a rotation $R_h$ of angle $h$ in this circle (in the direction from $\xi$ to ${\bf e}$) leaving the orthogonal complement in $\R^{d+1}$ invariant, and write $\zeta - h := R_h(\zeta)$. We proceed in a similar way to \cite[Lemma 9]{g-remanuel} . The idea is to look at the following quantity
\begin{align}\label{Eq_1_Lem_two}
\begin{split}
\lim_{h \to 0} \frac{(r+h)^{\beta}\intav{\mathcal{B}(\zeta-h,r+h)} f-r^{\beta}\intav{\mathcal{B}(\zeta,r)}f}{h} =&\lim_{h \to 0}\frac{(r+h)^{\beta}\intav{\mathcal{B}(\zeta-h,r+h)} f-r^{\beta}\intav{\mathcal{B}(\zeta-h,r+h)} f}{h}+\\
&\lim_{h\to 0}\frac{r^{\beta}\intav{\mathcal{B}(\zeta-h,r+h)} f- r^{\beta}\intav{\mathcal{B}(\zeta-h,r)} f}{h} +\\&\lim_{h\to 0}  \frac{r^{\beta}\intav{\mathcal{B}(\zeta-h,r)} f -  r^{\beta}\intav{\mathcal{B}(\zeta,r)}f}{h} .
\end{split} 
 \end{align}
In principle we do not know that the limit above exists. We shall prove that it in fact exists using the right-hand side of \eqref{Eq_1_Lem_two}. Once this is established, the left-hand side of \eqref{Eq_1_Lem_two} tells us that this limit must be zero, since the numerator is always nonpositive regardless of the sign of $h$. First we know that 
\begin{align}\label{1}
\lim_{h \to 0}\frac{(r+h)^{\beta}\intav{\mathcal{B}(\zeta-h,r+h)} f-r^{\beta}\intav{\mathcal{B}(\zeta-h,r+h)} f}{h}=\beta r^{\beta-1}\intav{\mathcal{B}(\zeta,r)}f.
\end{align}
From Lemma \ref{Lem4 - inner integral sphere} and (3.10) in  \cite[Lemma 9]{g-remanuel}, we note that
 \begin{align}\label{Eq_2_Lem_two}
 \lim_{h \to 0} \frac{r^{\beta}\intav{\mathcal{B}(\zeta-h,r)} f -  r^{\beta}\intav{\mathcal{B}(\zeta,r)} f }{h}= \nabla \widetilde{\mathcal{M}}_{\beta}f(\xi)(v(\xi, {\bf e})).
 \end{align}
Also, from (3.11) and (3.12) of the proof of \cite[Lemma 9]{g-remanuel}, we know that 
\begin{align}\label{2}
\begin{split}
\lim_{h\to 0}\frac{r^{\beta}\intav{\mathcal{B}(\zeta-h,r+h)} f- r^{\beta}\intav{\mathcal{B}(\zeta-h,r)} f}{h} &=-r^{\beta}\frac{\sigma'(r)}{\sigma(r)} \intav{\mathcal{B}(\zeta,r)} \nabla f (\eta) (v(\eta, \zeta) )\,\frac{\sigma(d(\zeta,\eta))}{\sigma'(d(\zeta,\eta))}  \,\d\sigma(\eta)\\&=-r^{\beta}\frac{\sigma'(r)}{\sigma(r)}\left(\intav{\mathcal{B}(\zeta,r)}f-\intav{\partial \mathcal{B}(\zeta,r)}f\right)
.
\end{split}
\end{align}
We conclude combining \eqref{Eq_1_Lem_two},\eqref{1}, \eqref{Eq_2_Lem_two} and \eqref{2}.
\end{proof}
\subsubsection{Estimates from \cite{g-remanuel}}
We also need other estimates that we used in the proof of  \cite[Lemma 12]{g-remanuel}, we recall them here. For $\xi \in \mathbb{S}^d$ and $\mathcal{B}(\zeta,r)$ with $\zeta$ in the same big circle than ${\bf e}$ and $\xi$ and $0\le \theta(\zeta)=\theta(\xi)-r$,  let us define, for every $\eta \in \mathcal{B}(\zeta,r)$  $$S(\eta) = (\eta\cdot v(\xi, {\bf e}))\xi - (\eta\cdot \xi)v(\xi, {\bf e})$$ and $v_1(\eta)=\frac{S(\eta)}{|S(\eta)|}$.
\begin{lemma}\label{Step1}
There exists some universal constant $\rho>0$ such that if $\mathcal{B}(\zeta,r)\subset \overline{\mathcal{B}(\bf{e},\rho)}$ we have:
\begin{align*}
    \Big| \frac{\sigma'(r)}{\sigma(r)} \intav{\mathcal{B}(\zeta,r)} &\nabla f (\eta) (-v(\eta, \zeta) )\,\frac{\sigma(d(\zeta,\eta))}{\sigma'(d(\zeta,\eta))}  \,\d\sigma(\eta) + \, \intav{\mathcal{B}(\zeta,r)} \nabla f(\eta)  \frac{\theta(\zeta)}{r}\,v_1(\eta)\, \d \sigma(\eta)  \Big|\\ &\lesssim_{d}
    \intav{\mathcal{B}(\zeta,r)} |\nabla f (\eta)|\,\theta(\zeta) \,\d\sigma(\eta) + \frac{1}{r} \intav{\mathcal{B}(\zeta,r)} |\nabla f (\eta)|\,\theta(\eta) \,\d\sigma(\eta).
\end{align*}
\end{lemma}
\begin{proof}
This is (3.21) in \cite[Lemma 12]{g-remanuel}.
\end{proof}

\begin{lemma}\label{step2}
There exists a universal constant $\rho>0$ such that if $\mathcal{B}(\zeta,r)\subset \overline{\mathcal{B}({\bf{e}},\rho)}$, we have:
\begin{align*}
    \left| \intav{\mathcal{B}(\zeta,r)} \nabla f(\eta) \,\big(|S(\eta)| - 1\big) \frac{\theta(\zeta)}{r}\,v_1(\eta)\, \d \sigma(\eta)\right|
    \lesssim_{d} \intav{\mathcal{B}(\zeta,r)} \big|\nabla f(\eta)\big| \,\theta(\zeta)\, \d \sigma(\eta).
\end{align*}
\end{lemma}
\begin{proof}
This is (3.22) in \cite[Lemma 12]{g-remanuel}.
\end{proof}
\subsubsection{Estimates for large radii - main lemma} Now we prove an important estimate.

\begin{lemma}\label{crucial_lemma_large_radii} Let $\xi \in \mathbb{S}^d \setminus \{{\bf e}, {\bf -e}\}$ and let $\mathcal{B}(\zeta,r)  \in {\bf B}_{\xi}^{\beta}$, with $\zeta$ in the half great circle determined by ${\bf e}$, $\xi$ and $-{\bf e}$. Assume that $0 \leq \theta(\zeta) < \theta(\xi)$, that $\xi \in \partial \mathcal{B}(\zeta,r)$ and that $\widetilde{\mathcal{M}}_{\beta}f$ is differentiable at $\xi$. There is a universal constant $\rho >0 $ such that if $\mathcal{B} = \mathcal{B}(\zeta,r) \subset \overline{\mathcal{B}({\bf e},\rho)}$ then
\begin{align}\label{Est_crucial_lemma}
\big| \nabla \widetilde{\mathcal{M}}_{\beta}f(\xi)\big| \lesssim_d \frac{r^{\beta}}{\theta(\xi)} \intav{\mathcal{B} } |\nabla f (\eta)| \,\theta(\eta)\,\d\sigma(\eta) +  \frac{r^{\beta+1} \,\theta(\zeta)}{\theta(\xi)} \intav{\mathcal{B}} |\nabla f (\eta)|\,\d\sigma(\eta).
\end{align}
\end{lemma}
\begin{proof}
In the following we choose $\rho$ such that both Lemma \ref{Step1} and \ref{step2} hold. From Lemma \ref{Lem_10_prep} (and considering that $\big|\nabla \widetilde{\mathcal{M}}f(\xi)\big|=\nabla \widetilde{\mathcal{M}}f(\xi)(v(\xi,{\bf e}))$) we have
\begin{align}\label{Rep_1}
\nabla \widetilde{\mathcal{M}}_{\beta}f(\xi)(-v(\xi, {\bf e})) = r^{\beta}\frac{\sigma'(r)}{\sigma(r)} \intav{\mathcal{B}} \nabla f (\eta) (-v(\eta, \zeta) )\,\frac{\sigma(d(\zeta,\eta))}{\sigma'(d(\zeta,\eta))}  \,\d\sigma(\eta)+\beta r^{\beta-1}\intav{\mathcal{B}}f.
\end{align}
In the case $\zeta = {\bf e}$, estimate \eqref{Est_crucial_lemma} follows directly from \eqref{Rep_1} and \cite[Lemma 11]{g-remanuel} (this is just the smoothness of the function $\frac{t\sigma'(t)}{\sigma(t)}$ near $0$). From now on we assume that $\zeta \neq  {\bf e}$. From Lemma \ref{Lem4 - inner integral sphere} we also know that
\begin{align}\label{Rep_2}
\nabla \widetilde{\mathcal{M}}_{\beta}f(\xi)(-v(\xi, {\bf e})) = r^{\beta}\intav{\mathcal{B}} \nabla f(\eta) (S(\eta)) \,\d \sigma(\eta),
\end{align}
with $S$ defined as in the previous section. 
The idea is to compare the identities \eqref{Rep_1} and \eqref{Rep_2} in order to bound $\big|\nabla \widetilde{\mathcal{M}}_{\beta}f(\xi)\big| = \big|\nabla \widetilde{\mathcal{M}}_{\beta}f(\xi)(-v(\xi, {\bf e}))\big |$. To do so, we write the right-hand side of \eqref{Rep_2} as a sum of three terms, one being comparable to $\big|\nabla \widetilde{\mathcal{M}}_{\beta}f(\xi)\big|$, the second one being small, and the third one being close to the right-hand side of \eqref{Rep_1} in a suitable sense. We start by writing
$$1 = \frac{\theta(\xi) - \theta(\zeta)}{r} = \frac{d({\bf e},\xi) - d({\bf e},\zeta)}{r}.$$ 
We then have
\begin{align}\label{rep_3}
\begin{split}
r^{\beta}\intav{\mathcal{B}} \nabla f(\eta) \,S(\eta) \,\d \sigma(\eta) & = r^{\beta}\intav{\mathcal{B}} \nabla f(\eta) \,|S(\eta)|  \left(\frac{\theta(\xi) - \theta(\zeta)}{r}\right)v_1(\eta)\, \d \sigma(\eta)\\
& = r^{\beta}\intav{\mathcal{B}} \nabla f(\eta) \,|S(\eta)| \, \frac{\theta(\xi)}{r}\,v_1(\eta)\, \d \sigma(\eta) \\
&  \ \ \  \ - r^{\beta}\intav{\mathcal{B}} \nabla f(\eta) \,\big(|S(\eta)| - 1\big) \frac{\theta(\zeta)}{r}\,v_1(\eta)\, \d \sigma(\eta) - r^{\beta}\intav{\mathcal{B}} \nabla f(\eta)  \frac{\theta(\zeta)}{r}\,v_1(\eta)\, \d \sigma(\eta).
\end{split}
\end{align}
By Lemma \ref{Step1} we have
\begin{align*}
\begin{split}
& \left|\frac{\sigma'(r)}{\sigma(r)} \intav{\mathcal{B}} \nabla f (\eta) (-v(\eta, \zeta) )\,\frac{\sigma(d(\zeta,\eta))}{\sigma'(d(\zeta,\eta))}  \,\d\sigma(\eta) +\, \intav{\mathcal{B}} \nabla f(\eta)  \frac{\theta(\zeta)}{r}\,v_1(\eta)\, \d \sigma(\eta) \right| \\
&  \ \ \ \ \ \ \ \ \ \ \lesssim_d \intav{\mathcal{B}} |\nabla f (\eta)|\,\theta(\zeta) \,\d\sigma(\eta) + \frac{1}{r} \intav{\mathcal{B}} |\nabla f (\eta)|\,\theta(\eta) \,\d\sigma(\eta).
\end{split}
\end{align*}
So, we have  
\begin{align}\label{05172019_14:52}
\begin{split}
& \left|r^{\beta}\frac{\sigma'(r)}{\sigma(r)} \intav{\mathcal{B}} \nabla f (\eta) (-v(\eta, \zeta) )\,\frac{\sigma(d(\zeta,\eta))}{\sigma'(d(\zeta,\eta))}  \,\d\sigma(\eta) +\, r^{\beta-1}\intav{\mathcal{B}} \nabla f(\eta) \theta(\zeta)\,v_1(\eta)\, \d \sigma(\eta) \right| \\
&  \ \ \ \ \ \ \ \ \ \ \lesssim_d r^{\beta}\intav{\mathcal{B}} |\nabla f (\eta)|\,\theta(\zeta) \,\d\sigma(\eta) + r^{\beta-1} \intav{\mathcal{B}} |\nabla f (\eta)|\,\theta(\eta) \,\d\sigma(\eta).
\end{split}
\end{align}
Also, by Lemma \ref{step2} we have:
\begin{align}\label{05172019_14:51}
\left |r^{\beta}\intav{\mathcal{B}} \nabla f(\eta) \,\big(|S(\eta)| - 1\big) \frac{\theta(\zeta)}{r}\,v_1(\eta)\, \d \sigma(\eta)\right|\lesssim_{d} r^{\beta} \intav{\mathcal{B}} \big|\nabla f(\eta)\big| \,\theta(\zeta)\, \d \sigma(\eta).
\end{align}
We notice that
\begin{align}\label{rep4}
    \begin{split}
    - \frac{\theta(\xi)}{r} \big|\nabla \widetilde{\mathcal{M}}_{\beta}f(\xi) \big |&=r^{\beta}\intav{\mathcal{B}} \nabla f(\eta) \,|S(\eta)| \, \frac{\theta(\xi)}{r}\,v_1(\eta)\, \d \sigma(\eta)\\
    &=r^{\beta}\frac{\sigma'(r)}{\sigma(r)} \intav{\mathcal{B}} \nabla f (\eta) (-v(\eta, \zeta) )\,\frac{\sigma(d(\zeta,\eta))}{\sigma'(d(\zeta,\eta))}  \,\d\sigma(\eta)+\beta r^{\beta-1}\intav{\mathcal{B}}f\\
    &+r^{\beta}\intav{\mathcal{B}} \nabla f(\eta) \,\big(|S(\eta)| - 1\big) \frac{\theta(\zeta)}{r}\,v_1(\eta)\, \d \sigma(\eta) \\
    &+r^{\beta}\intav{\mathcal{B}} \nabla f(\eta)  \frac{\theta(\zeta)}{r}\,v_1(\eta)\, \d \sigma(\eta),
    \end{split}
\end{align}
where the last equality is obtained by comparing identities \eqref{Rep_1}, \eqref{Rep_2} and \eqref{rep_3}.
So, combining \eqref{05172019_14:52}, \eqref{05172019_14:51} and \eqref{rep4}, we get
\begin{align*}
    \frac{\theta(\xi)}{r}\big |\nabla \widetilde{\mathcal{M}}_{\beta}f(\xi) \big|&\le r^{\beta}\frac{\sigma'(r)}{\sigma(r)} \intav{\mathcal{B}} \nabla f (\eta) (v(\eta, \zeta) )\,\frac{\sigma(d(\zeta,\eta))}{\sigma'(d(\zeta,\eta))}  \,\d\sigma(\eta)\\
    &+\left|r^{\beta}\intav{\mathcal{B}} \nabla f(\eta) \,\big(|S(\eta)| - 1\big) \frac{\theta(\zeta)}{r}\,v_1(\eta)\, \d \sigma(\eta)\right|-r^{\beta}\intav{\mathcal{B}} \nabla f(\eta)  \frac{\theta(\zeta)}{r}\,v_1(\eta)\, \d \sigma(\eta)\\
    &\lesssim_d r^{\beta}\intav{\mathcal{B}} |\nabla f (\eta)|\,\theta(\zeta) \,\d\sigma(\eta) + r^{\beta-1} \intav{\mathcal{B}} |\nabla f (\eta)|\,\theta(\eta) \,\d\sigma(\eta).
\end{align*}
And finally
\begin{align*}
   \big| \nabla \widetilde{\mathcal{M}}_{\beta}f(\xi)\big| \lesssim_d \frac{r^{\beta}}{\theta(\xi)} \intav{\mathcal{B} } |\nabla f (\eta)| \,\theta(\eta)\,\d\sigma(\eta) +  \frac{r^{\beta+1} \,\theta(\zeta)}{\theta(\xi)} \intav{\mathcal{B} } |\nabla f (\eta)|\,\d\sigma(\eta). 
\end{align*}
This concludes the proof of the lemma.
\end{proof}
\begin{subsubsection}{Estimates for small radii}
We also need another estimate, similar to the one obtained in \cite[Lemma 2.10]{boundednessradialfrac}. 

Given a ball $\mathcal{B}=\mathcal{B}(\zeta,r)$ we define $2\mathcal{B}=\mathcal{B}(\zeta,2r)$. We use the following estimate (the analogous in the polar case to \cite[Proposition 2.8]{boundednessradialfrac}), its verification is left to the interested reader:
\begin{proposition}\label{lemageometrico}
    Suppose that $g\in L^{1}(\mathbb S^{d})$ is polar, $\mathcal{B}:=\mathcal{B}(\zeta,r)\subset \mathbb{S}^{d}\setminus \mathcal{B}({\bf{e}},2r)\cup \mathcal{B}(-{\bf{e}},2r)$,  then we have that 
\begin{align*}
    \intav{[\theta(\zeta)-r,\theta(\zeta)+r]}|g|\lesssim_{d}\intav{2\mathcal{B}
    (\zeta ,r)}|g|,
\end{align*}
where in the first integral we consider the one dimensional function corresponding to g.
\end{proposition}
We also need the following proposition. We say that $\mathcal{B}(\zeta,r)\subset \mathbb{S}^d$, with $r\le \pi$, is a best ball for $\widetilde{\mathcal{M}}_{\beta}f$, if there exists $\xi \in \mathcal{B}(\zeta,r)$ with $\widetilde{\mathcal{M}}_{\beta}f(\xi)=r^{\beta}\intav{\mathcal{B}(\zeta,r)}f$.
\begin{proposition}\label{best}
Suppose that $0<\beta<d$, $f\in L^1(\mathbb{S}^d)$, $\mathcal{B}_1:=\mathcal{B}(\zeta _1,r_1)$ and $\mathcal{B}_2=\mathcal{B}(\zeta_2,r_2)$ are best balls for $\widetilde{\mathcal{M}}_{\beta}f$ such that $\mathcal{B}_2\subset \mathcal{B}(\zeta_1,cr_1)$ with $c>1$, then we have that:
\begin{align*}
    \left(\frac{r_1}{r_2}\right)^{\beta}\intav{\mathcal{B}_1}f\lesssim_{c,d} \intav{\mathcal{B}_2}f.
\end{align*}
\begin{proof}
The proof is analogous to the proof of \cite[Proposition 2.11]{boundednessradialfrac}, by using the fact that $1 \lesssim_{c,d} \frac{\sigma(r_1)}{\sigma(cr_1)}$.
\end{proof}
\end{proposition}
Now, we define $w(\xi):=\text{min}\{\theta(\xi),\pi-\theta(\xi)\}$. Then, we have the following result. 
\begin{lemma}\label{lemamadridluiro}
Suppose that $f\in W^{1,1}(\mathbb{S}^{d})$ is polar, $0<\beta<d$, $\mathcal{B}\in \bf{B}^{\beta}_{\xi}$ for some $\xi\in \mathbb{S}^d$, $\mathcal{B}=\mathcal{B}(\zeta,r)$, $r\le \frac{w(\zeta)}{4}$ and 
\begin{align*}
    \mathcal{E}:=\left \{\eta \in 2\mathcal{B}: \frac{1}{2} \intav{\mathcal{B}}f\le f(\eta) \le 2\intav{\mathcal{B}}f \right\}.
\end{align*}
Then 
\begin{align*}
    \left| \intav{\mathcal{B}}\nabla f(\eta)S(\eta)\d \sigma(\eta) \right |\lesssim_{d,\beta} \intav{2\mathcal{B}} |\nabla f(\eta)|\chi_{\mathcal{E}}(\eta)d\sigma(\eta).
\end{align*}
\end{lemma}
\begin{proof}
We know by Lemma \ref{Lem_10_prep} and \ref{Lem4 - inner integral sphere} that:
\begin{align*}
   \left| \intav{\mathcal{B}}\nabla f(\eta)S(\eta)\d \sigma(\eta) \right |\leq \frac{\sigma'(r)}{\sigma(r)}\left(\intav{\mathcal{B}} f-\intav{\partial \mathcal{B}}f \right). 
\end{align*}
Let us define  $a:=\theta(\zeta)-r$, $b:=\theta(\zeta)+r$ and 
\begin{align*}
    A:=\left\{t\in 2[a,b]:\frac{1}{2} \intav{\mathcal{B}}f\le f(t) \le 2\intav{\mathcal{B}}f\right\}. 
\end{align*}
Now we show that 
\begin{align*}
    \intav{\mathcal {B}}f-\intav{\partial \mathcal{B}}f\leq 2\int_{[a,b]}|f'(t)|\chi_{A}(t)\d t
\end{align*}
in an analogous way to  \cite[Lemma 2.10]{boundednessradialfrac}. We conclude using that $\left|\nabla f(\eta)\right|\chi_{\mathcal E}(\eta)=\left|f'(\theta(\eta))\right|\chi_{A}(\theta(\eta))$ for $\eta \in \mathcal{E}$, Proposition \ref{lemageometrico} and the fact that $\frac{r\sigma'(r)}{\sigma(r)}$ is bounded.
\end{proof}
\end{subsubsection}
\subsubsection{Proof of Theorem \ref{boundedness}-Lipschitz case} We are now in position to move on to the proof of Theorem \ref{boundedness}
when our initial datum $f$ is a Lipschitz function. In this case we also have $\widetilde{\mathcal{M}}_{\beta}f$ Lipschitz. 

For each $\xi \in \mathbb{S}^{d} \setminus \{{\bf e}, {\bf -e}\}$ let us choose a ball $\mathcal{B}_{\xi}:=\mathcal{B}(\zeta_{\xi},r_{\xi}) \in {\bf B}_{\xi}^{\beta}$ with $r_{\xi}$ minimal and, subject to this condition, with $\zeta_{\xi}$ in the half great circle connecting ${\bf e}, \xi, {\bf -e}$ in a way that $w(\zeta_{\xi}) = \min \{d({\bf e},\zeta_\xi), d({-\bf e},\zeta_\xi)\}$ is minimal. If there are two potential choices for $\zeta_{\xi}$ we choose the one with $0 \leq \theta(\zeta_{\xi}) \leq \theta(\xi).$ 
\smallskip
\begin{proof}[Proof of Theorem \ref{boundedness}, Lipschitz case]
First let us observe that  by Lemma \ref{Lem4 - inner integral sphere} we have:
\begin{align}\label{computation}
\begin{split}
    \int_{\mathbb{S}^d}|\nabla \widetilde{\mathcal{M}}_{\beta}f |^{q}&=\int_{\mathbb{S}^d} \left |r_{\xi}^{\beta}\intav{\mathcal{B}_{\xi}}\nabla f(\eta)S(\eta)\d\sigma(\eta)\right|^{q}\d \sigma(\xi)\\
    &=\int_{\mathbb{S}^d}\frac{r_{\xi}^{q\beta}}{\mathcal{\sigma}(r_{\xi})^{q-1}}\left |\int_{\mathcal{B}_{\xi}}\nabla f(\eta)S(\eta)\d\sigma(\eta)\right|^{q-1}\left|\intav{\mathcal{B}_{\xi}}\nabla f(\eta)S(\eta)\d \sigma(\eta) \right|\d \sigma(\xi)\\
    &\lesssim_{d,\beta} {\|\nabla f\|_1}^{q-1}\int_{\mathbb{S}^d}\left |\intav{\mathcal{B}_{\xi}}\nabla f(\eta)S(\eta)\d \sigma(\eta)\right |\d \sigma(\xi),
\end{split}
\end{align}
where we use the fact that $q\beta=d(q-1)$ and that $\frac{r^{d}}{\sigma(r)}$ is bounded. So we need to bound the integral term. This is done in four steps.

\noindent{\it Step 1:}
 Let us observe that we can restrict our attention to small balls. Define the set $\mathcal{R}_c = \left\{\xi \in \mathbb{S}^{d} \ :  \xi \in \mathbb S^{d}\setminus{\{{\bf e},-{\bf e}\}}  \ {\text {and}} \  r_{\xi} \geq c\right\}$.We find that
\begin{align}\label{small}
\int_{\mathcal{R}_c}\left |\intav{\mathcal{B}_{\xi}}\nabla f(\eta)S(\eta)\d\eta\right|\d \sigma(\xi) & \leq  \int_{\mathcal{R}_c} \,\frac{1}{\sigma(\mathcal{B}_{\xi})}\int_{\mathcal{B}_{\xi}} |\nabla f(\eta)| \,\d \sigma(\eta)\, \d \sigma(\xi) \lesssim_{c,d} \int_{\mathbb{S}^d} |\nabla f(\eta)|\,\d \sigma(\eta).
\end{align}

\noindent{\it Step 2:} Let us define $\mathcal{W}_d=\{\xi \in \mathbb S^{d};r_{\xi}\le \frac{w(\xi)}{4}\}$. We show that we can restrict our attention to $\xi\in \mathbb{S}^{d}\setminus {\mathcal{W}_d}$. For this, we use Lemma  \ref{lemamadridluiro}. For every $\xi \in \mathcal{W}_d$ we define:
\begin{align*}
\mathcal{A}_{\xi}:=\left\{\eta\in 2\mathcal{B}_{\xi}:\frac{1}{2}\intav{\mathcal{B}_{\xi}}f\le f(\eta)\le 2\intav{\mathcal{B}_{\xi}}f\right\}. 
\end{align*}
So, by Lemma \ref{lemamadridluiro} we have:
\begin{align*}
 \left| \intav{B_{\xi}}\nabla f(\eta)S(\eta)\d \sigma(\eta) \right |\lesssim_{d,\beta} \intav{2B_{\xi}} |\nabla f(\eta)|\chi_{\mathcal{A_{\xi}}}(\eta)\d\sigma(\eta),
\end{align*}
 therefore:
\begin{align*}
    \int_{\mathcal{W}_{d}}\left |\intav{B_{\xi}}\nabla f(\eta)S(\eta)\d\sigma(\eta)\right |\d\sigma (\xi)&\lesssim_{d,\beta} \int_{\mathcal{W}_d} \intav{2B_{\xi}} |\nabla f(\eta)|\chi_{\mathcal{A_{\xi}}}(\eta)\d\sigma(\eta)\d \sigma(\xi)\\
    &\lesssim_{d,\beta}\int_{\mathbb{S}^d}|\nabla f(\eta)|\left (\int_{\mathbb {S}^{d}}\frac{\chi_{2\mathcal{B}_{\xi}}(\eta)\chi_{\mathcal{A}_{\xi}}(\eta)\chi_{\mathcal{W}_{d}}(\xi)}{\sigma(2\mathcal{B}_{\xi})}\d \sigma(\xi)\right )\d \sigma(\eta).  
\end{align*}
We want to bound the inner integral for fixed $\eta \in \mathbb{S}^{d}$. Now suppose that $\chi_{\mathcal{A}_{\xi_1}}(\eta)\neq 0$ and $\chi_{\mathcal{A}_{\xi_2}}(\eta)\neq 0$ for some $\xi_1,\xi_2\in \mathbb{S}^{d}$. 
If these points do not exist, the estimates are obvious. By definition, the above means that $\frac{1}{2}\intav{B_{\xi_1}}f\le f(\eta) \le 2\intav{B_{\xi_1}}f$ and $\frac{1}{2}\intav{B_{\xi_2}}f\le f(\eta) \le 2\intav{B_{\xi_2}}f$.
In particular, we have $$\frac{1}{4}\intav{B_{\xi_1}}f\le \intav{B_{\xi_2}}f\le 4\intav{B_{\xi_1}}f.$$
Let $r_1:=rad(\mathcal{B}_{\xi_1})$ and $r_2:=rad(\mathcal{B}_{\xi_2})$. First, assume $r_2\le r_1$. Since $\eta \in 2\mathcal{B}_{\xi_2}\cap 2\mathcal{B}_{\xi_1}$ it follows that $\mathcal{B}_{\xi_2}\subset 8\mathcal{B}_{\xi_1}$. And then, by Proposition \ref{best}:
\begin{align*}
    \left(\frac{r_1}{r_2}\right)^{\beta}\intav{\mathcal{B}_{\xi_1}}f\lesssim_{d} \intav{\mathcal{B}_{\xi_2}}f\lesssim_{d} \intav{\mathcal{B}_{\xi_1}}f,
\end{align*}
then it follows that $r_1\lesssim_{\beta,d} r_2$. And then, by symmetry, we have  
\begin{align*}
\frac{r_1}{r_2} \simeq_{d,\beta}1   
\end{align*}
and that implies that if $\eta\in \mathcal{A}_{\xi} $ then 
$d(\xi,\eta)\lesssim_{d,\beta} rad(\mathcal{B}_{\xi_1})$ and $\sigma(\mathcal{B}_{\xi_1})\lesssim_{d,\beta} \sigma(\mathcal{B}_{\xi})$. Combining these estimates we have the following
\begin{align*}
    \int_{\mathbb {S}^{d}}\frac{\chi_{2\mathcal{B}_{\xi}}(\eta)\chi_{\mathcal{A}_{\xi}}(\eta)\chi_{\mathcal{W}_d}(\eta)}{\sigma(2\mathcal{B}_{\xi})}\d \sigma(\xi)\lesssim_{d,\beta} \int_{B(\eta,C(d,\beta)rad(\mathcal{B}_{\xi_1}))}\frac{\d \xi}{\sigma(\mathcal{B}_{\xi_{1}})}\lesssim_{d,\beta} 1.
\end{align*}
From where we have that 
\begin{align}\label{estimativaLM}
    \int_{\mathcal{W}_{d}}\left |\intav{B_{\xi}}\nabla f(\eta)S(\eta)\d\sigma(\eta)\right |\d\sigma (\xi)\lesssim_{d,\beta} \|\nabla f\|_1.
\end{align}    
So, we need to prove a similar estimate for the remaining points. Using \eqref{small} we can see that we may restrict ourselves to the situation where $d({\bf e},\xi)\leq \rho$ or $d({\bf -e},\xi) \leq \rho$ (where $\rho$ is given by Lemma \ref{crucial_lemma_large_radii}), we can do that because there exist $r_{\rho}$ such that if $r\le r_{\rho}$ and $\mathcal{B}(\zeta,r)\in {\bf B}_{\xi}^{\beta}$ then $w(\xi)\le \rho$ or $\xi\in \mathcal{W}_d$. By symmetry let us assume that $\theta(\xi) = d({\bf e},\xi)\leq \rho$. Then we define the set 
\begin{align*}
 \mathcal{G}_{d}=\left\{\xi\notin \mathcal{W}_{d}\cup \mathcal{R}_{r_{\rho}}: \xi\in \mathcal{B}(\bf e,\rho)\right\},    
\end{align*}
 and further decompose it in $\mathcal{G}_d^- = \left \{\xi \in \mathcal{G}_d  \ : \ 0 \leq \theta(\zeta_{\xi}) < \theta(\xi) \right \}$ and $\mathcal{G}_d^+ = \{\xi \in \mathcal{G}_d  \ : \ 0 < \theta(\xi) < \theta(\zeta_{\xi})\}$. We bound the integrals over these two sets separately, as done in \cite{g-remanuel}.

\noindent {\it Step 3 (Bounding the integral on $\mathcal{G}_{d}^+$)}. For $\mathcal{G}_d^+$ we proceed as follows.
\begin{align*}
\begin{split}
 \int_{\mathcal{G}_d^+}\left|\intav{\mathcal{B}_{\xi}}\nabla f(\eta)S(\eta)\d \sigma(\eta)\right|\d \sigma(\xi) &\le\int_{\mathcal{G}_d^+} \, \intav{\mathcal{B}_{\xi}} |\nabla f(\eta)| \,\d\sigma(\eta) \,\d \sigma(\xi)\\
& = \int_{\mathbb{S}^d} |\nabla f(\eta)| \int_{\mathcal{G}_d^+} \frac{\chi_{\mathcal{B}_{\xi}}(\eta)}{\sigma(\mathcal{B}_{\xi})}\,\d\sigma(\xi)\,\d\sigma(\eta).
\end{split}
\end{align*}
Note that $\theta(\eta) \geq \theta(\xi)$ in this case. So, as in Step 1 of the proof of \cite[Theorem 2, Lipschitz case]{g-remanuel}  we get $$\int_{\mathcal{G}_d^+}\frac{\chi_{\mathcal{B}_{\xi}}(\eta)}{\sigma(\mathcal{B}_{\xi})}\d \sigma(\xi)\lesssim_{d} 1,$$ and conclude that
\begin{align}\label{Step3}
  \int_{\mathcal{G}_d^+}\left|\intav{\mathcal{B}_{\xi}}\nabla f(\eta)S(\eta)\d \sigma(\eta)\right|\d \sigma(\xi)\lesssim_{d} \|\nabla f\|_1.   
\end{align}

\smallskip

\noindent {\it Step 4 (Bounding the integral on $\mathcal{G}_{d}^-$ ) }. We now bound the integral over $\mathcal{G}_d^-$ using Lemma
\ref{crucial_lemma_large_radii}. If $\xi \in \mathcal{G}_d^-$ we then have
\begin{align}\label{bigine}
\begin{split}
& \int_{\mathcal{G}_d^-} \left| \intav{B_{\xi}}\nabla f(\eta)S(\eta)\d \sigma(\eta)\right| \,\d \sigma(\xi) \lesssim_d \int_{\mathcal{G}_d^-}  \left(\frac{1}{\theta(\xi)} \intav{\mathcal{B}_{\xi} } |\nabla f (\eta)| \,\theta(\eta)\,\d\sigma(\eta) +  \frac{r_{\xi} \,\theta(\zeta_{\xi})}{\theta(\xi)} \intav{\mathcal{B}_{\xi}} |\nabla f (\eta)|\,\d\sigma(\eta) \right) \d \sigma(\xi)\\
  &\ \  \lesssim \int_{\mathbb{S}^d}  |\nabla f (\eta)| \int_{\mathcal{G}_d^-} \frac{\chi_{\mathcal{B}_{\xi}}(\eta) \, \theta(\eta)}{r_{\xi}\,\sigma(r_{\xi})}\,\d \sigma(\xi)\,\d \sigma(\eta) +  \int_{\mathbb{S}^d}  |\nabla f (\eta)| \int_{\mathcal{G}_d^-} \frac{\chi_{\mathcal{B}_{\xi}}(\eta) \, \theta(\zeta_{\xi})}{\sigma(r_{\xi})}\,\d \sigma(\xi)\,\d \sigma(\eta).
 \end{split}
\end{align}
Now, we notice, as in (3.28) and (3.29) from the proof of \cite[Theorem 2, Lipschitz case]{g-remanuel}, that
\begin{align}\label{cota1}
\int_{\mathcal{G}_d^-}\frac{\chi_{\mathcal{B}_{\xi}}(\eta)\theta(\zeta)}{\sigma(\mathcal{B}_{\xi})}\d \sigma(\xi)\lesssim_{d} 1    \end{align}
and
\begin{align}\label{cota2}
 \int_{\mathcal{G}_d^-}\frac{\chi_{\mathcal{B}_{\xi}}(\eta)\theta(\eta)}{r_{\xi}\sigma(\mathcal{B}_{\xi})}\d \sigma(\xi)\lesssim_{d} 1.
\end{align}
Our desired inequality 
\begin{align}\label{Step 4}
  \int_{\mathcal{G}_d^-}\left|\intav{\mathcal{B}_{\xi}}\nabla f(\eta)S(\eta)\d \sigma(\eta)\right|\d \sigma(\xi)\lesssim_{d} \|\nabla f\|_1   
\end{align} 
follows combining \eqref{bigine}, \eqref{cota1} and \eqref{cota2}. 
Then, by combining \eqref{computation}, \eqref{small}, \eqref{estimativaLM}, \eqref{Step3} and \eqref{Step 4} we conclude Theorem \ref{boundedness} in this case. 
\end{proof}
\subsection{Passage to the general case}

In this subsection we develop some results that are useful in the proof of Theorem \ref{continuity}, so our preliminaries are more general than needed for Theorem \ref{boundedness}.
\begin{subsubsection}{Preliminaries of the reduction}
\begin{lemma}\label{modulo} Let $f\in W^{1,1}(\mathbb{S}^{d})$ be such that $\|f-f_j\|_{W^{1,1}}\rightarrow 0$ as $j\rightarrow \infty$. Then $\||f_j|-|f|\|_{W^{1,1}}\rightarrow 0$ as $j\rightarrow \infty$. 
\end{lemma}
\begin{proof}
The proof is exactly as  \cite[Lemma 2.3]{continuityfractionalradial}.
\end{proof}
\begin{lemma}\label{pointwiseconvergence}
Let $f\in W^{1,1}_{\text{pol}}(\mathbb{S}^{d})$ and $\{f_j\}_{j\in \mathbb{N}}\subset W^{1,1}_{\text{pol}}(\mathbb{S}^d)$ be such that $\|f_j-f\|_{W^{1,1}}\rightarrow 0 $ as $j\rightarrow \infty$. Then
\begin{align*}
    \nabla \widetilde{\mathcal{M}_{\beta}}f_j(\xi)\rightarrow \nabla \widetilde{\mathcal{M}_{\beta}}f(\xi)
\end{align*} a.e. as $j\rightarrow \infty.$

\end{lemma}
\begin{proof}
By Lemma \ref{modulo} we can assume that the functions $f$ and $f_j$ are nonnegative. We consider the set $E\subset [0,\pi]$ that consists of the points $\theta(\xi)$ where $\widetilde{\mathcal{M}}_{\beta}f,\widetilde{\mathcal{M}}_{\beta}f_j$ are all differentiable at $\xi$. By Lemma \ref{Lem4 - inner integral sphere} and the almost everywhere differentiability of $\widetilde{\mathcal{M}}_{\beta}f$ and $\widetilde{\mathcal{M}}_{\beta}f_j$,  we have that  $m(E^c)=0$ and that $$\nabla \widetilde{\mathcal{M}_{\beta}}f_j(\xi)(-v(\xi,{\bf e}))=r_{\xi,j}^{\beta}\intav{\mathcal{B}(\zeta_{\xi,j},r_{\xi,j})}\nabla f_j(\eta)S(\eta)\d \sigma(\eta),$$ for every $\xi\in \mathbb{S}^d$ with $\theta(\xi) \in E$. So, we just need to prove that $$\underset{j\rightarrow \infty}{\lim}r_{\xi,j}^{\beta}\intav{\mathcal{B}(\zeta_{\xi,j},r_{\xi,j})}\nabla f_j(\eta)S(\eta)\d \sigma(\eta)=r_{\xi}^{\beta}\intav{\mathcal{B}(\zeta_{\xi},r_{\xi})} \nabla f(\eta)S(\eta)\d \sigma(\eta).$$
Let us assume that there exists $\varepsilon>0$ and $(j_k)_{k\in \mathbb{N}}$ such that 
\begin{align}\label{inequality}
\left |r_{\xi,j_k}^{\beta}\intav{B(\zeta_{\xi,j_k},r_{\xi,j_k})}\nabla f_{j_k}(\eta)S(\eta)\d \sigma(\eta)-r_{\xi}^{\beta}\intav{\mathcal{B}(\zeta_{\xi},r_{\xi})}\nabla f(\eta)S(\eta)\d \sigma(\eta)\right|>\varepsilon.
\end{align}
Then, by compactness, there exists a subsequence of $ (j_k)_{k\in \mathbb{N}}$ (we write this subsequence also by $ (j_k)_{k\in \mathbb{N}}$) such that $\underset{k\rightarrow \infty}{\lim}r_{\xi,j_k}=r_0$ and $\underset{k\rightarrow \infty}{\lim}\zeta_{\xi,j_k}=\zeta_0$. By Lemma \ref{radiusconvergence} we conclude that $r_0>0$ and that $\mathcal{B}(\zeta_0,r_0)\in \mathbf{B}_{\xi}^{\beta}$ for almost every $\xi$ with $\theta(\xi)\in E$, so we have that 
\begin{align*}
    \underset{k\rightarrow \infty}{\lim}r_{\xi,j_k}^{\beta}\intav{\mathcal{B}_{\xi,j_k}} \nabla f_{j_k}(\eta)S(\eta)\d \sigma(\eta)=r_{0}^{\beta}\intav{\mathcal{B}(\zeta_0,r_0)}\nabla f(\eta)S(\eta)\d \sigma(\eta),
\end{align*}
reaching a contradiction with \eqref{inequality}. This concludes the proof of the lemma.
\end{proof}
We can conclude, in a similar way, the following proposition.
\begin{proposition}
\label{pointwisemaximalconvergence}
If $f_j\rightarrow f$ in $W^{1,1}_{\text{pol}}(\mathbb{S}^d)$ and $0<\beta<d$, we have that
\begin{align*}
\underset{j\rightarrow \infty}{\lim}\widetilde{\mathcal{M}}_{\beta}f_{j}(\xi)=\widetilde{\mathcal{M}}_{\beta}f(\xi)
\end{align*}
for a.e. $\xi\in \mathbb{S}^d$ .

\end{proposition}
\end{subsubsection}
Now we conclude the passage to the general case:
\begin{subsubsection}{Proof of general case of Theorem \ref{boundedness}}
\begin{proof}[Proof of Theorem \ref{boundedness}]Consider a sequence $f_n\in W^{1,1}_{\text{pol}}(\mathbb{S}^{d})$ with $f_n\ge 0$ Lipschitz and $\|f_n-f\|_{W^{1,1}(\mathbb{S}^d)}\rightarrow 0$. By Fatou's lemma, Lemma \ref{pointwiseconvergence} and Theorem \ref{boundedness} in the Lipschitz case we conclude:
\begin{align*}
\|\nabla \widetilde{\mathcal{M}}_{\beta}f\|_{q}\le \underset{n\rightarrow \infty}{\liminf}\|\nabla \widetilde{\mathcal{M}}_{\beta}f_n\|_{q}\lesssim_{d,\beta} \underset{n\rightarrow \infty}{\lim}\|\nabla f_n\|_1=\|\nabla f\|_1. 
\end{align*}
This concludes the proof of the theorem in the general case.
\end{proof}
We need, in the next section, one consequence of the proof of Theorem \ref{boundedness}.
\begin{corollary}\label{localboundedness}
Let $f\in W^{1,1}_{\text{pol}}(\mathbb{S}^d)$. Let $\mathcal{A}\subset \mathbb{S}^{d}$ measurable such that there exists an open ball $\mathcal{B}({\bf e},\gamma)\subset \mathbb{S}^d$ with $\gamma\le \frac{\pi}{4}$ and $\mathcal{B}_{\xi}\subset \mathcal{B}({\bf e},\gamma)$ for some $\mathcal{B}_{\xi}\in \bf{B}_{\xi}^{\beta}$, for all $\xi \in \mathcal{A}$. Then we have
$$\int_{\mathcal{A}}\big|\nabla \widetilde{\mathcal{M}}_{\beta}f\big|^q\lesssim_{d,\beta}\left(\int_{\mathcal{B}({\bf e},2\gamma)}|\nabla f|\right)^q.$$
\end{corollary}
\begin{proof}
We prove slightly modified versions of Lemma \ref{Lem_10_prep} and Lemma \ref{crucial_lemma_large_radii}, where we do the additional assumption that for every $\mathcal{B}(\zeta,r)\in {\bf B}_{\xi}^{\beta}$ we have $0\le \theta(\zeta)<\theta(\xi)$, these versions hold in the $W^{1,1}_{\text{rad}}(\mathbb{S}^d)$ case. The proofs of these modified versions follows by approximating by Lipschitz functions and using Lemma \ref{radiusconvergence}. Then, in order to conclude the corollary, we follow the proof of Theorem \ref{boundedness} (without assuming $f$ Lipschitz). 
\end{proof}
\end{subsubsection}
\section{Proof of Theorem \ref{continuity} }
We notice that if $f$ is the constant $0$ the results presented in this section are direct consequences of the boundedness, so we assume that this is not the case.
In this section we choose $\mathcal{B}(\zeta,r)=\mathcal{B}_{\xi}\in \bf{B}_{\xi}^{\beta}$ with maximal radius, in case of two possibilities we choose the one with the center closer to ${\bf e}$. We also define the analogous for every $f_j$. Given Lemma \ref{modulo}, we may assume that $f$ and $f_j$ are nonnegative. Also, we assume that $f$ and $f_j$ are continuous in $\mathbb{S}^d\setminus{\{{\bf{e}},-{\bf{e}}\}}$ (we can assume this after possibly modifying the functions in a set of measure zero). Given a compact $\mathcal{K}\subset \mathbb{S}^d$ with ${\bf{e}},-{\bf{e}}\notin \mathcal{K}$ we define $d_{\mathcal{K}}:=d(\{{\bf{e}},-{\bf{e}}\},\mathcal{K})$.   
\subsection{Convergence inside a compact set far from the poles}
Let us prove first an useful proposition.
\begin{proposition}\label{radiusneedstobebig}
If $f\in W^{1,1}(\mathbb{S}^d)$ is polar we have that, for every compact $\mathcal{K}\in \mathbb{S}^d$ with $d_{\mathcal{K}}>0$, there exists $\rho_{\mathcal{K},f}>0$, such that $r_{\xi}>\rho_{\mathcal{K},f}$ for every $\xi \in \mathcal{K}$.  
\end{proposition}
\begin{proof}
Let us define $\mathcal{K}_1=\left\{\xi \in \mathcal{K};r_{\xi}\ge \frac{d_{K}}{4}\right\}$ and $\mathcal{K}_2=\left\{\xi \in \mathcal{K};r_{\xi}< \frac{d_{K}}{4}\right\}$. We  can see that $$\mathcal{B}_{\xi}\subset \mathbb{S}^{d}\setminus \left\{ \mathcal{B}({\bf e},\frac{d_{\mathcal{K}}}{4})\cup \mathcal{B}(-{\bf e},\frac{d_{\mathcal{K}}}{4})\right\}$$ for every $\xi \in \mathcal{K}_2$. 

Let us define $$N:=N_{f,\mathcal{K}}=\sup_{\xi\in \mathbb{S}^{d}\setminus \left\{ \mathcal{B}({\bf e},\frac{d_{K}}{4})\cup \mathcal{B}(-{\bf e},\frac{d_{K}}{4})\right\}}f(\xi),$$ we know $N<\infty$ by the continuity of $f$ in that set. We can assume that $N>0$ (because if not, we have that $\mathcal{K}_2$ is empty and the proposition follows directly). We can see that $$\widetilde {\mathcal{M}}_{\beta}f(\xi)\le r_{\xi}^{\beta}N$$ for $\xi \in \mathcal{K}_2$ , but we also know that $$\frac{\pi^{\beta}}{\sigma(\mathbb{S}^d)}\|f\|_1\le \widetilde{ \mathcal{M}}_{\beta}f(\xi),$$ so, in particular, we have that $$0<c:=\left(\frac{\pi^{\beta}}{\sigma(\mathbb{S}^d)N}\|f\|_1\right)^{\frac{1}{\beta}}\le r_{\xi}.$$
Taking $\rho_{\mathcal{K},f}<\text{min}\{c,\frac{d_{K}}{4}\}$ the proposition follows.
\end{proof}
\begin{lemma}\label{continuityofrho}
If $f_j\to f$ in $W^{1,1}_{\text{pol}}(\mathbb{S}^d)$ and $d_{\mathcal{K}}>0$, we have that there exists $\rho:=\rho (\mathcal{K},f,(f_j)_{j\in \mathbb{N}})>0$ such that for every $j$ big enough we have $$r_{\xi,j}>\rho$$ for all $\xi \in \mathcal{K}$.
\end{lemma}
\begin{proof}
Since $f_j\to f$ in $W^{1,1}_{\text{pol}}(\mathbb{S}^d)$ and $d_{\mathcal{K}}>0$, we can conclude that $f_j\to f$ uniformly in $K$. In particular for $j$ big enough $N_{f,K}\sim N_{f_j,K}$ and $\|f\|_1 \sim \|f_j\|_1$, so the lemma follows by the proof of Proposition \ref{radiusneedstobebig}. 
\end{proof}

Using this result we can prove the following key proposition.
\begin{proposition}\label{convergenceincompact}
If $f_j\to f$ in $W^{1,1}_{\text{pol}}(\mathbb{S}^d)$ we have that for every $\mathcal{K}$ compact with $d_{\mathcal{K}}>0$, we have $$|\nabla \widetilde{\mathcal{M}}_{\beta}f_j|\rightarrow |\nabla \widetilde{\mathcal{M}}_{\beta}f| $$ in $L^{q}(\mathcal{K}).$ 
\end{proposition}
\begin{proof}
Given Proposition \ref{continuityofrho}, as in computation \eqref{computation} we have that, for $j$ big enough and almost every $\xi \in \mathcal{K}$:  
\begin{align*}
|\nabla \widetilde{\mathcal{M}}_{\beta}f_j|^q(\xi) &\lesssim_{d,\beta} \|\nabla f_j\|_{1}^{q-1}\left | \intav{\mathcal{B}_{\xi}}\nabla f_j(\eta)S(\eta)\d \sigma(\eta)\right| \\
&\lesssim_{d,\beta} \|\nabla f_j\|_{1}^{q-1} \frac{\|\nabla f_j\|_1}{\sigma(\rho)}. \end{align*}
Since the last term is uniformly bounded (given that $\|\nabla f_j\|_1\rightarrow \|\nabla f\|_1$ and $\rho>0$) by the Dominated convergence theorem and Lemma \ref{pointwiseconvergence} we conclude. 

\end{proof}

\subsection{Smallness outside the compact}
The next proposition implies the required control outside our compact $\mathcal{K}$.

\begin{proposition}\label{smallness}
If  $f_j\rightarrow f$ in $W^{1,1}_{\text {pol}}(\mathbb{S}^d)$. Given $\varepsilon >0$, there exists $\mathcal{K}$ compact with $d_{\mathcal{K}}>0$  such that for $j$ big enough we have
$$\int_{\mathcal{K}^{c}}\big | \nabla \widetilde{\mathcal{M}}_{\beta}f_j\big |^q<\varepsilon $$
and that $$\int_{\mathcal{K}^{c}}\big | \nabla \widetilde{\mathcal{M}}_{\beta}f\big |^q<\varepsilon. $$
\end{proposition}
\begin{proof}
Let us take $\mathcal{K}=\mathbb{S}^d\setminus \left\{ \mathcal{B}({\bf e},d_{\mathcal{K}})\cup \mathcal{B}(-{\bf e},d_{\mathcal{K}})\right\}$ compact such that \begin{align}\label{epsilon}\int_{\mathcal{B}({\bf e},2(2l+1)d_{\mathcal{K}})\cup \mathcal{B}(-{\bf e},2(2l+1)d_{\mathcal{K}})}|\nabla f|\le \frac{\varepsilon_1}{2},
\end{align} where $l$ is a constant to be defined later (notice that $\mathcal{K}$ depends on $l$), we have then:
$$\int_{\mathcal{B}({\bf e},2(2l+1)d_{\mathcal{K}})\cup \mathcal{B}(-{\bf e},2(2l+1)d_{\mathcal{K}})}|\nabla f_j|\le \varepsilon_1,$$
for $j$ big enough. 

We write $\mathcal{B}({\bf e},2(2l+1)d_{\mathcal{K}})\cup \mathcal{B}(-{\bf e},2(2l+1)d_{\mathcal{K}})=:\Omega_{\mathcal{K}}(l)$.  
Let us define $\mathcal{E}_1=\{ \xi \in \mathcal{K}^c;r_{\xi}\ge ld_{\mathcal{K}}\}$ and $\mathcal{E}_2=\{ \xi \in \mathcal{K}^c;r_{\xi}<ld_{\mathcal{K}}\}$, when we replace $r_{\xi}$ by $r_{\xi,j}$ we call the subsets $\mathcal{E}_1^{j}$ and $\mathcal{E}_2^j$. 
Let us see that, as in computation \eqref{computation},

$$\int_{\mathcal{E}_1}\big |\nabla \widetilde{\mathcal{M}}_{\beta}f\big |^q\lesssim_{d,\beta} \|\nabla f\|_{1}^{q-1} \frac{\|\nabla f\|_1}{\sigma(ld_{\mathcal{K}})}\sigma(d_{\mathcal {K}}).$$
Now we deal with the elements  in $\mathcal{E}_2$, by symmetry is enough to deal with the terms in $\mathcal{B}({\bf e} ,d_{\mathcal{K}})\cap \mathcal{E}_2$, here by Corollary \ref{localboundedness} and \eqref{epsilon}, we have that
\begin{align*}
\int_{\mathcal{B}({\bf e},d_{\mathcal{K}})\cap \mathcal{E}_2}|\nabla \widetilde{\mathcal{M}}_{\beta}f|^q\lesssim_{d,\beta} \left(\int_{\Omega_{\mathcal{K}}(l)}|\nabla f|\right)^q\le \varepsilon_1^q.   
\end{align*}
So, we have $$\int_{\mathcal{K}^c}|\nabla \widetilde{\mathcal{M}}_{\beta}f|^q\lesssim_{d,\beta} (\varepsilon_1)^q+ \|\nabla f\|_{1}^{q}  \frac{\sigma(d_{\mathcal {K}})}{\sigma(ld_{\mathcal{K}})},$$
and then taking $\varepsilon_1$ small and $l$ big enough we get the required bound related to $f$. The bounds related with $f_j$ when $j$ is big enough are obtained analogously using that $\|\nabla f\|_1\sim \|\nabla f_j\|_1$, and then we can choose a compact $\mathcal{K}$ such that the result holds.
\end{proof}
\subsection{Proof of Theorem \ref{continuity}}
\begin{proof}[Proof of Theorem \ref{continuity}]Follows directly by combining Propositions \ref{convergenceincompact} and \ref{smallness}.
\end{proof}
In a similar way we can conclude the following conjectural result. 
\begin{theorem}\label{conjecturalresult}
Let $0<\beta<d$ and $q=\frac{d}{d-\beta}$. Assume that a suitable analogue of Corollary \ref{localboundedness} holds for the centered fractional Hardy-Littlewood operator in $\mathbb{S}^d$. Then the map $f\mapsto |\nabla \mathcal{M}_{\beta}f|$ is continuous from $W^{1,1}_{\text{pol}}(\mathbb{S}^d)$ to $L^q(\mathbb{S}^d)$. 
\end{theorem}
\section{Continuity for radial functions in $\mathbb{R}^d$}
As already discussed in Subsection \ref{continuityintro}, the method developed in the previous section allows to obtain analogous results in the euclidean case. We briefly discuss these implications in this section.
\subsection{Uncentered case}
We present an alternative proof of \cite[Theorem 1.2]{continuityfractionalradial}.
\begin{theorem}\label{continuityradial}
Let $0<\beta<1 $ and $q=\frac{d}{d-\beta}$. We have that the map $f\mapsto |\nabla \widetilde{M}_{\beta}f|$ is continuous from $W^{1,1}_{\text{rad}}(\mathbb{R}^d)$
to $L^q(\mathbb{R}^d)$.
\end{theorem}
In order to conclude this theorem we need two local boundedness results, the first one in a neighborhood of $0$ and the second one in a neighborhood of $\infty$. They are both corollaries of the proof of \cite[Theorem 1.1]{boundednessradialfrac},  we omit their proof.
\begin{lemma}\label{localboundednes0}
Let $f\in W^{1,1}_{\text{rad}}(\mathbb{R}^d)$. Let  $H\in \mathbb{R}^d$ be such that for every $x\in H$ there exists $B_x\in \bf{B}_{x}^{\beta}$  (where these definitions are analogous to the ones in the sphere) with $B_x\subset B(0,\rho)$, for some $\rho >0$. Then we have  
$$\int_{H}|\nabla \widetilde{M}_{\beta}f|^q\lesssim_{\beta,d} \left(\int_{B(0,2\rho)}|\nabla f|\right)^q.$$
\end{lemma}
\begin{lemma}\label{localboundednessinfty}
Let us take $f\in W^{1,1}_{\text{rad}}(\mathbb{R}^d)$. Let $G\subset \mathbb{R}^d$ be such that for every $x\in G$ there exists $B_x\in \bf{B}_{x}^{\beta}$  with $B_x\subset B(0,\rho)^{c}$, for some $\rho>0$. Then we have 
$$\int_{G}|\nabla \widetilde{M}_{\beta}f|^q\lesssim_{\beta,d} \left(\int_{B(0,\frac{\rho}{2})^{c}}|\nabla f|\right)^q.$$
\end{lemma}
The proof of the following proposition is analogous to the proof of Proposition \ref{convergenceincompact}, using the pointwise convergence in this case \cite[Lemma 2.4]{continuityfractionalradial}. We omit it. 
\begin{proposition}\label{continuitycompactrd}
If $f_j\rightarrow f$ in $W^{1,1}_{\text{rad}}(\mathbb{R}^d)$ and $0<a<b<\infty$ we have that
$$|\nabla \widetilde{M}_{\beta}f_j|\rightarrow |\nabla \widetilde{M}_{\beta}f| $$
in $L^{q}(K_{a,b})$, where $K_{a,b}:=\{x\in \mathbb{R}^d; a\le |x|\le b\}$. 
\end{proposition}

\begin{proof}[Proof of Theorem \ref{continuityradial}.]
Given any $\varepsilon>0$ we claim that there exists $a>0$ such that for $j$ big enough we have
$$\int_{B(0,a)}|\nabla \widetilde{M}_{\beta}f_j|^q<\varepsilon $$
and that $$\int_{B(0,a)}|\nabla \widetilde{M}_{\beta}f|^q<\varepsilon. $$ This is done as in Proposition \ref{smallness} using Lemma \ref{localboundednes0}. Also, as is done in \cite[Proposition 4.10]{continuityfractionalradial}, we conclude that there exist $0<b<\infty$, with $a<b$, such that, for $j$ big enough $$\int_{B(0,b)^c}|\nabla \widetilde{M}_{\beta}f_j|^q<\varepsilon$$ and $$\int_{B(0,b)^c}|\nabla \widetilde{M}_{\beta}f|^q<\varepsilon.$$ We notice that in this step we need Lemma \ref{localboundednessinfty}.
Thus, combining these observations with Proposition \ref{continuitycompactrd}, we conclude the proof.
\end{proof}

\subsection{Centered case}
We present here our conjectural result concerning the centered Hardy-Littlewood fractional maximal operator, the proof is analogous to the proof of Theorem \ref{continuityradial}, we omit it.
\begin{theorem}
Let $0<\beta<1$, $q=\frac{d}{d-\beta}$. Assume that Lemma \ref{localboundednes0} and \ref{localboundednessinfty} hold for the centered Hardy-Littlewood fractional maximal operator $M_{\beta}$. Then, the map
$$f\mapsto |\nabla M_{\beta}f|$$ is continuous from $W^{1,1}_{\text{rad}}(\mathbb{R}^d)$ to $L^q(\mathbb{R}^d)$.
\end{theorem}
Notice that our assumptions are stronger that just assuming the boundedness of the map, they are not direct implications of this. In the uncentered case, for instance, they were consequences of the proof of the boundedness. We expect these results (Lemma \ref{localboundednes0} and Lemma \ref{localboundednessinfty} for the centered case) to be consequences of a suitable proof of the boundedness for radial functions in the centered case. 

\section*{Acknowledgments}
The author would like to thank to Emanuel Carneiro, for the encouragement and very useful comments, and the anonymous referees for corrections and several valuable
suggestions that improved the presentation. The author acknowledges support from CAPES-Brazil. Support from the STEP programme is gratefully acknowledged. 
\bibliographystyle{plain}
\bibliography{name5}
\end{document}